\documentclass[a4paper,11pt,oneside,bookmarks]{article}
\usepackage[utf8x]{inputenc}
\usepackage{graphicx}
\usepackage{wrapfig}
\usepackage[pdftex,bookmarks]{hyperref}
\hypersetup{
pdftitle={On characterization of locally $L^0$-convex topologies induced by a family of $L^0$-seminorms},
pdfauthor={José M. Zapata}
}

\usepackage{fancyhdr}
\usepackage[english]{babel}
\usepackage{amsmath}
\usepackage{amsthm}
\usepackage{amsfonts}

\usepackage{amscd}

\usepackage{epigraph}

\usepackage{amssymb}

\newcommand{\N}{\mathbb{N}}
\newcommand{\R}{\mathbb{R}}

\newtheorem{defn}{Definition}[section]

\newtheorem{prop}{Proposition}[section]

\newtheorem{thm}{Theorem}[section]

\newtheorem{example}{Example}[section]

\providecommand{\keywords}[1]{\textbf{\textit{Keywords:}} #1}

\begin{document}


\title{On characterization of locally $L^0$-convex topologies induced by a family of $L^0$-seminorms}

\author{José M. Zapata}

\date{\today}
\maketitle


\begin{abstract}
The purpose of this paper is to provide a characterization of locally $L^0$-convex modules induced by a family of $L^0$-seminorms using the gauge function for $L^0$-modules. Taking advantage of these ideas we will give a counterexample of a locally $L^0$-convex module whose topology is not induced by a family of $L^0$-seminorms.
\end{abstract}

 

\section*{Introduction}

In \cite{key-1}, motivated by the financial applications, Filipovic, Kupper and Vogelpoth try to provide an appropriate theoretical framework in order to study the conditional risk measures and develop the classical convex analysis for topological $L^0$-modules.

To this end, they introduce the gauge function for $L^0$-modules and, in the same way as in the convex analysis, they claim that a topological $L^0$-module is locally $L^0$-convex if and only if its topology is induced by a family of $L^0$-seminorms.

Nevertheless, in \cite{key-2} T. Guo, S. Zhao and X. Zeng warn that there is a hole in the proof and introduce some theoretical considerations.

In this paper, we go further and provide a characterization  of locally $L^0$-convex modules induced by a family of $L^0$-seminorms.
Finally, taking advantage of these ideas, we will give a counterexample of a locally $L^0$-convex module whose topology cannot be induced by any family of $L^0$-seminorms.
   
\section{Some  important concepts}

Given a probability space $\left(\Omega,\mathcal{F},P\right)$, which will be fixed for the rest of this paper, we consider the set $L^{0} \left(\Omega,\mathcal{F},P\right)$, the set of equivalence classes of  real valued $\mathcal{F}$-measurable random variables, which will be denoted simply as $L^{0}$.

It is known that the triple $\left(L^{0},+,\cdot\right)$ endowed with the partial order of the almost sure dominance is a lattice ordered ring.

We say ``$X\geq Y$`` if $P\left( X\geq Y \right)=1$.
Likewise, we say ``$X>Y$'', if $P\left( X> Y \right)=1$. 

And, given $A\in \mathcal{F}$, we say that $X>Y$ (respectively,  $X \geq Y$) on $A$, if $P\left(X>Y \mid A\right)=1$ (respectively , if $P\left(X \geq Y \mid  A \right)=1$).

We also define 

\[
L_{+}^{0}:=\left\{ Y\in L^{0};\: Y\geq 0\right\} 
\]

\[
L_{++}^{0}:=\left\{ Y\in L^{0};\: Y>0  \right\}. 
\]

We can also define the set $\bar{L^{0}}$, the set of equivalence classes of  $\mathcal{F}$-measurable random variables taking values in $\bar{\R}=\R\cup\{\pm\infty\}$, and extend the partial order of the almost sure dominance to $\bar{L^{0}}$. 

~\\

Let us see below, some notions and results that will be used in the development of this paper

~\\

In A.5 of \cite{key-3} is proved the proposition below

\begin{prop}
Let $\phi$  be a subset of $L^{0}$, then
\begin{enumerate}
\item There exists $Y^{*}\in\bar{L^{0}}$ such that $Y^{*}\geq Y$ for all 
$Y\in\phi,$ and such that any other $Y'$ satisfying the same, verifies $Y'\geq Y^{*}.$
\item Suppose that  $\phi$ is directed upwards. Then there exists a non-decreasing sequence  $Y_{1}\leq Y_{2}\leq...$ in $\phi,$ such that 
 $Y_{n}$ converges to $Y^{*}$ almost surely.
\end{enumerate}
\end{prop}

\begin{defn}
Under the conditions of the previous proposition, 
$Y^{*}$ is called essential supremum of $\phi$, and we write

\[
ess.sup\,\phi=\underset{Y\in\phi}{ess.sup\, Y}:=Y^{*}
\]

The essential infimum of $\phi$ is defined as

\[
ess.inf\,\phi=\underset{Y\in\phi}{ess.inf\, Y}:=\underset{Y\in\phi}{-ess.sup\,\left(-Y\right)}
\]
\end{defn}

The order of the almost sure dominance also lets us define a topology on $L^{0}$. Let us define

\[
B_{\varepsilon}:=\left\{ Y\in L^{0};\:\left|Y\right|\leq\varepsilon\right\} 
\]

the ball of radius $\varepsilon\in L_{++}^{0}$ centered at $0\in L^{0}$.
Then, for all $Y\in L^{0}$, $\mathcal{\mathcal{U}}_{Y}:=\left\{ Y+B_{\varepsilon};\:\varepsilon\in L_{++}^{0}\right\} $ is a neighborhood base of $Y$. Thus, it can be defined a topology on $L^{0}$ that it will be known as the topology induced by $\left|\cdot\right|$ and $L^{0}$ endowed with this topology will be denoted by $L^{0}\left[\left|\cdot\right|\right]$. 

\begin{defn}
A topological $L^{0}$-module $E\left[\tau\right]$ is a $L^{0}$-module $E$ endowed with a topology $\tau$ such that 
\begin{enumerate}
\item $E\left[\tau\right]\times E\left[\tau\right]\longrightarrow E\left[\tau\right],\left(X,X'\right)\mapsto X+X'$
and
\item $L^{0}\left[\left|\cdot\right|\right]\times E\left[\tau\right]\longrightarrow E\left[\tau\right],\left(Y,X\right)\mapsto {Y}X$
\end{enumerate}
are continuous with the corresponding product topologies.
\end{defn}

\begin{defn}
A topology $\tau$ on a $L^{0}$-module $E$ is a locally $L^{0}$-convex module 
 if there is a neighborhood base of $0\in{E}$ $\mathcal{U}$ such that each $U\in\mathcal{U}$ is
\begin{enumerate}
\item $L^{0}$-convex, i.e. ${Y}X_{1}+{(1-Y)}X_{2}\in U$ for all $X_{1},X_{2}\in U$
and $Y\in L^{0}$ with $0\leq Y\leq1,$
\item $L^{0}$-absorbent, i.e. for all $X\in E$ there is a $Y\in L_{++}^{0}$
such that $X\in {Y}U,$
\item $L^{0}$-balanced, i.e. ${Y}X\in U$ for all $X\in U$ and $Y\in L^{0}$
with $\left|Y\right|\leq1.$

In this case, $E\left[\tau\right]$ is a locally $L^0$-convex module. 
\end{enumerate}
\end{defn}

\begin{defn}
A function $\left\Vert \cdot\right\Vert :E\rightarrow L_{+}^{0}$
is a $L^{0}$-seminorm on $E$ if:
\begin{enumerate}
\item $\left\Vert {Y}X\right\Vert =\left|Y\right|\left\Vert X\right\Vert $
for all $Y\in L^{0}$ y $X\in E.$
\item $\left\Vert X_{1}+X_{2}\right\Vert \leq\left\Vert X_{1}\right\Vert +\left\Vert X_{2}\right\Vert ,$
for all $X_{1},X_{2}\in E.$

If, moreover

\item $\left\Vert X\right\Vert =0$ implies $X=0,$
\end{enumerate}
Then $\left\Vert \cdot\right\Vert $ is a $L^{0}$-norm on
$E$

\end{defn}
\begin{defn}
Let $\mathcal{P}$ be a family of $L^0$-seminorms on a $L^{0}$-module
$E$. Given $Q\subset\mathcal{P}$ finite and $\varepsilon\in L_{++}^{0},$
we define 

\[
U_{Q,\varepsilon}:=\left\{ X\in E;\:\underset{\left\Vert .\right\Vert \in Q}{\sup}\left\Vert X\right\Vert \leq\varepsilon\right\} .
\]
Then for all $X\in E$, $\mathcal{\mathcal{U}}_{Q,X}:=\left\{ X+U_{\varepsilon};\:\varepsilon\in L_{++}^{0},\: Q\subset \mathcal{P}\: {finite} \right\} $
is a neighborhood base of $X$. Thereby, we define a topology
on $E$, which it will be known as the topology induced by $\mathcal{P}$
and $E$ endowed with this topology will be denoted by $E\left[\mathcal{P}\right]$.

Furthermore, it is proved by the lemma 2.16 of \cite{key-1}
 that $E\left[\mathcal{P}\right]$ is a locally $L^0$-convex module.
\end{defn}

\section{The gauge function and the countable concatenation closure.}

Let us write the notion of gauge function given in \cite{key-1}:  

\begin{defn}
Let $E$ be a $L^0$-module. The gauge function $p_{K}:E\rightarrow\bar{L}_{+}^{0}$ of
a set $K\subset E$ is defined by 

\[
p_{K}\left(X\right):=ess.inf\left\{ Y\in L_{+}^{0};\: X\in YK\right\} .
\]

\end{defn}

In addition, in \cite{key-1} the properties below are proved:

\begin{prop}
\label{prop: gaugeConAbs}
The gauge function $p_{K}$ of a $L^{0}$-convex and $L^{0}$-absorbent
$K\subset E$ satisfies:
\begin{enumerate}
\item $1_{A}p_{K}\left(1_{A}X\right)=1_{A}p(X),$ for all $A\in\mathcal{F}$
and $X\in E$.
\item $p_{K}\left(X\right)=ess.inf\left\{ Y\in L_{++}^{0};\: X\in Y{K}\right\} $
for all $X\in E$.
\item ${Y}p_{K}(X)=p_{K}(Y{X})$ for all $X\in E$ and $Y\in L_{+}^{0}$
\item $p_{K}(X+Y)\leq p_{K}(X)+p_{K}(Y)$ for $X,\, Y\in E$.
\item For all $X\in E$ there exists a sequence $\left\{ Z_{n}\right\} $
in $L_{++}^{0}$ such that $Z_{n}\searrow p_{K}\left(X\right)$ almost
surely and such that $X\in Z_{n}K$ for all $n$.
\item If in addition, $K$ is $L^{0}$-balanced then $p_{K}\left(Y{X}\right)=\left|Y\right|p_{K}\left(X\right)$
for all $Y\in L^{0}$ and $X\in E$.
\end{enumerate}
In particular, $p_{K}$ is an $L^{0}$-seminorm.
\end{prop}

We also have the next result (see 2.4 of \cite{key-1} and 2.22 of \cite{key-2}):

\begin{prop}
The gauge function $p_{U}$ of a $L^{0}$-convex and $L^{0}$-absorbent set $U\subset$ satisfies: 
\begin{enumerate}
\item $p_{U}(X)\geq1$ on $B$ for all $X\in E$ with $1_{A}X\notin1_{A}U$
for all $A\in\mathcal{F}$ with $P(A)>0$, $A\subset B$.
\item If in addition, $E$ is a locally $L^{0}$-convex module, then
\[
\overset{0}{U}\subset\left\{ X\in E;\: p_{U}(X)<1\right\} 
\]
\end{enumerate}
\end{prop}

Proceeding in the same way as the classical convex analysis, given a $L^{0}$-convex, $L^{0}$-absorbent and $L^{0}$-balanced set $U\subset E$, one can expect that $\left\{ X\in E;\: p_{U}(X)<1\right\} \subset U$ holds. If this held, we could prove that any topological $L^0$-module is locally $L^0$-convex module if and only if its topology is induced by a family of $L^0$-seminorms.

Not in vain, this statement is set as valid in theorem 2.4 given in \cite{key-1}.

However, in \cite{key-3} the authors point out that the proof of theorem 2.4 given in \cite{key-1} has a hole and conjecture that, according to their observations, in  general the topology of a locally $L^0$-convex module is not necessary induced by a family of $L^0$-seminorms, but no counterexample is given.

We go further and provide an example (see \ref{contraejemplo}) of a locally $L^0$-convex module, whose topology is not induced by any family of $L^0$-seminorms. Therefore, the theorem 2.4 given in \cite{key-1} does not hold.

In addition, this example shows that there exists a $L^{0}$-convex, $L^{0}$-absorbent and $L^{0}$-balanced set $U\subset E$ such that $\left\{ X\in E;\: p_{U}(X)<1\right\} \nsubseteq U$.

~\\

Let us introduce some notation:

Given a $L^0$-module $E$, we denote by $\Pi(\Omega,\mathcal{F})$ the set of countable partitions on $\Omega$ to $\mathcal{F}$. 


Let $E$ be a $L^{0}$-module. Given a set $C\subset E$, we call the countable concatenation closure of $C$ the set 

\[
\overline{C}^{\Pi}:=\{X\in E; \: \exists \{A_n\}_{n\in\N} \in \Pi(\Omega,\mathcal{F}) \textnormal{ with } 1_{A_n}X\in 1_{A_n}C  \}. 
\]

We say that $C$ is closed under countable concatenations on $E$,
if 
\[
C=\overline{C}^{\Pi}.
\]

\begin{prop}
Let $E\left[\tau\right]$ be a locally $L^0$-convex module and 
$U\subset E$ a $L^{0}$-convex, $L^{0}$-absorbent, $L^{0}$-balanced and closed under countable concatenations on $E$ set. Then

\[
\overset{{\scriptstyle 0}}{U}\subset\left\{ X\in E;\: p_{U}(X)<1\right\} \subset U\subset\left\{ X\in E;\: p_{U}\left(X\right)\leq1\right\} 
\]
\end{prop}

\begin{proof}
It suffices to show that $\left\{ X\in E;\: p_{U}(X)<1\right\} \subset U$. Indeed,  let $X\in E$ be such that $p_{U}(X)<1$. By proposition \ref{prop: gaugeConAbs} there exists a sequence $\left\{ Y_{n}\right\} _{n\in\mathbb{N}}$ in $L_{++}^{0}$ such that $X\in Y_{n}U$ and $Y_{n}\searrow p_{U}(X)$.
In this way, we consider the sequence of sets $A_{0}:=\phi,\, A_{n}:=\left(Y_{n}<1\right)-A_{n-1}$ for $n>0$. Thus, $A_{n\in\N}$ is a partition of $\Omega$ and we define $Y:=\underset{n\in\mathbb{N}}{\sum}Y_{n}1_{A_{n}}\in L_{++}^{0}$.
Then, for each $n\in\N$, $1_{A_n}\frac{X}{Y}=1_{A_n}\frac{X}{Y_n}\in 1_{A_n}U.$ 
Hence, $\frac{X}{Y}\in  \overline{U}^\Pi=U$ as $U$
is closed under countable concatenations on $E$. 

Thereby, it is fulfilled that $p_{U}(X)\leq Y\leq1$. Thus, the convexity of $U$ implies $X=Y\cdot\frac{X}{Y}+(1-Y)\cdot0\in U$.  
\end{proof}

In the theorem below, we provide a characterization of the topological $L^0$-modules whose topology is induced by a family of $L^0$-seminorms. This statement differs from the theorem  2.4 of \cite{key-1} in requiring an extra condition over the elements of the neighborhood base of $0\in{E}$, namely, being closed under countable concatenations on $E$. 

\begin{thm}
\label{thm: caracterizacion}
Let $E\left[\tau\right]$ be a topological $L^{0}$-module. Then $\tau$ is induced by a family of $L^0$-seminorms if and only if there is a neighborhood base of $0\in{E}$ for which each $U\in\mathcal{U}$ is 
\begin{enumerate}
\item $L^{0}$-convex,
\item $L^{0}$-absorbent,
\item $L^{0}$-balanced and
\item closed under countable concatenations on $E$, i.e, $U=\overline{U}^\Pi$.
\end{enumerate}
\end{thm}

\begin{proof}
Suppose that $\tau$ is induced by a family of $L^0$-seminorms. 
If $Q\subset\mathcal{P}$ is finite and $\varepsilon\in L_{++}^{0}$, by inspection follows that $B_{Q,\varepsilon}$ is $L^{0}$-convex, $L^{0}$-absorbent and $L^{0}$-balanced.
Besides, $B_{Q,\varepsilon}$ is closed under countable concatenations on $E$.
Indeed, if $X=\sum_{n}1_{A_{n}}X_{n}$ with $X_{n}\in B_{Q,\varepsilon}$
for all $n\in\mathbb{N}$ and $\left\{ A_{n}\right\} _{n\in\mathbb{N}}$
is a partition of $\Omega$ with $A_n\in\mathcal{F}$ it holds for
 $\left\Vert \cdot\right\Vert \in Q$ that
\[
\left\Vert X\right\Vert =\left(\sum_{n}1_{A_{n}}\right)\left\Vert X\right\Vert =\sum_{n}1_{A_{n}}\left\Vert X\right\Vert =
\]

\[
=\sum_{n}\left\Vert 1_{A_{n}}X\right\Vert =\sum_{n}\left\Vert 1_{A_{n}}X_{n}\right\Vert =\sum_{n}1_{A_{n}}\left\Vert X_{n}\right\Vert \leq\varepsilon.
\]

Reciprocally, let $\mathcal{U}$ be a neighborhood base of $0\in{E}$ for which each $U\in\mathcal{U}$ is 
$L^{0}$-convex, $L^{0}$-absorbent, $L^{0}$-balanced and closed under countable concatenations on $E$.
Let us consider the family of $L^0$-seminorms $\left\{ p_{U}\right\} _{U\in\mathcal{U}}$ and let us show that it induces the topology $\tau$.
 Given $U\in\mathcal{U}$ is clear that $U\subset U_{p_{U},1}$. Therefore, for $\varepsilon\in L_{++}^{0}$ there exists $U\text{'\ensuremath{\in\mathcal{U}}}$ such that $\frac{1}{\varepsilon}U'\subset U\subset U_{p_{U},1}$ due to the continuity of product.
Thus, $U'\subset U_{p_{U},\varepsilon}$. 
On the other hand, for $U\in\mathcal{U}$, it is holds that $U_{p_{U},\frac{1}{2}}\subset\left\{ X\in E;\: p_{U}(X)<1\right\} \subset U$.
\end{proof}

Taking advantage of the ideas of the last theorem, we provide an example of a locally $L^0$-convex module, whose topology is not induced by any family of $L^0$-seminorms.

\begin{example}
\label{contraejemplo}
Given a probabilistic space $(\Omega,\mathcal{F},P)$ and an infinite partition $\left\{ A_{n}\right\}_{n\in\mathbb{N}}$
of $\Omega$ with $A_{n}\in\mathcal{F}$ and $P(A_n)>0$ (for example, $\Omega=(0,1)$, $\mathcal{F}=B((0,1))$, $A_{n}=[\frac{1}{2^n},\frac{1}{2^{n-1}})$ with $n\in\mathbb{N}$ and $P$ the Lebesgue measure). 

Let $\varepsilon\in L_{++}^{0}$ be, we define the set 

\[
U_{\varepsilon}:=\left\{ Y\in L^{0};\:\exists\, I\subset\mathbb{N}\textnormal{ finite, }\left|Y1_{A_{i}}\right|\leq\varepsilon\:\forall\, i\in\mathbb{N}-I\right\}.
\]

Then, it is easily shown that $U_{\varepsilon}$ is
$L^{0}$-convex, $L^{0}$-absorbent and $L^{0}$-balanced, and
$\mathcal{U}:=\left\{ U_{\varepsilon};\:\varepsilon\in L_{++}^{0}\right\} $
is a neighborhood base of $0\in{E}$ which generates a topology for which $L^{0}$ is a topological $L^{0}$-module.

Furthermore, it holds that $U_{\varepsilon}$ is not closed under countable concatenations on $L^{0}$.

Indeed, it is verified that $\varepsilon+1\notin U_{\varepsilon}$, but
$\varepsilon+1=\sum_{n}\left(\varepsilon+1\right)1_{A_{n}}$ with $\left(\varepsilon+1\right)1_{A_{n}}\in U_{\varepsilon}$.

Easily, it can be shown that any neighborhood base of $0\in{E}$ generating the same topology verified that its elements are not closed under countable concatenations on $L^{0}$.

Therefore, due to theorem \ref{thm: caracterizacion}, $L^{0}$, endowed with the topology generated by $\mathcal{U}$, is a locally $L^0$-convex module, whose topology is not induced by any family of $L^0$-seminorms.

Besides, it has to be met that $\left\{ X\in L^{0};\: p_{U_{\varepsilon}}\left(X\right)<1\right\} \nsubseteq U_{\varepsilon}$
for some $\varepsilon\in L_{++}^{0}$.  Otherwise, the family of $L^0$-seminorms
 $\left\{ P_{U_{\varepsilon}}\right\}_{\varepsilon\in L_{++}^{0}}$ would induce the topology.

In fact, we claim that $p_{U}\left(X\right)=0$ for all $X\in{L^0}$ and $U\in\mathcal{U}$.
It suffices to show that $p_{U_1}\left(1\right)=0$, since $p_{U_1}$ is a $L^0$-seminorm. By way of contradiction, assume $p_{U_1}\left(1\right)>0$. Then, there exists $m\in\N$ 
such that $P\left[\left(p_{U_1}\left(1\right)>0\right)\cap{A_m}\right]>0$.
Define $A:=\left(p_{U_1}\left(1\right)>0\right)\cap{A_m}$, $\nu:=\frac{p_{U_1\left(1\right)}+1_{A^c}}{2}$, $Y:=1_{A^c}+\nu{1_A}$ and $X:=1_{A^c}+\frac{1}{\nu}{1_A}$. 
Thus, we have $1=Y{X}\in{YU_1}$ and $P\left(p_{U_1}\left(1\right)>Y\right)>0$. We have a contradiction. 

\end{example}

Theorem \ref{thm: caracterizacion} gives rise to give a new more restrictive definition of locally $L^0$-convex module.
 In the following, a locally $L^0$-convex module will be as definition below says.

\begin{defn}
A locally $L^0$-convex module is a $L^{0}$-module such that there is a neighborhood base of $0\in{E}$ for which each $U\in\mathcal{U}$ is 
\begin{enumerate}
\item $L^{0}$-convex,
\item $L^{0}$-absorbent,
\item $L^{0}$-balanced and
\item closed under countable concatenations on $E$.
\end{enumerate}
\end{defn}

Thus, under this new definition a $L^{0}$-module $E\left[\tau\right]$ is a locally $L^0$-convex module if and only if $\tau$ is induced by a family of $L^0$-seminorms.

\begin{prop}
\label{prop: caracterizacion}
A $L^0$-module $E\left[\tau\right]$ is locally $L^0$-convex if and only if $\tau$ is induced by a family of $L^0$-seminorms.
\end{prop}

\end{document}